\documentclass[12pt, reqno, final, a4paper]{amsart}
\usepackage[pagebackref,colorlinks,linkcolor=purple,citecolor=blue,urlcolor=blue,hypertexnames=true]{hyperref}
\usepackage[capitalize,nameinlink]{cleveref}

\usepackage{color, bm, amscd, tikz-cd, amssymb}
\makeatletter
\newcommand{\proofstep}[1]{%
  \par
  \addvspace{\medskipamount}
  \textit{#1\@addpunct{.}}\enspace\ignorespaces
}
\makeatother

\usepackage[T1]{fontenc}
\usepackage[utf8]{inputenc}
\usepackage{amsmath,amssymb,amsthm,mathtools}
\usepackage[margin=1.15in]{geometry}
\usepackage{microtype}

\theoremstyle{plain}
\newtheorem{theorem}{Theorem}[section]
\newtheorem{proposition}{Proposition}[section]
\newtheorem{lemma}[theorem]{Lemma}
\newtheorem{corollary}[theorem]{Corollary}

\theoremstyle{definition}

\theoremstyle{remark}

\DeclareMathOperator{\Ree}{Re}
\DeclareMathOperator{\Cl}{Cl}
\DeclareMathOperator{\tr}{tr}
\newcommand{\D}{\mathbb D}
\newcommand{\T}{\mathbb T}
\newcommand{\C}{\mathbb C}

\newcommand{\E}{\mathbb E}
\newcommand{\Borel}{\mathcal B}

\title[LIL for inner functions]{Law of iterated logarithm for inner functions}

\author{Poornendu Kumar}
\address{Faculty of Mathematics and Physics, University of Ljubljana, Slovenia}
\email{poornendu.kumar@fmf.uni-lj.si}
\author{Raghavendra Tripathi}
\address{Department of Mathematics, Division of Science, New York University, Abu Dhabi, UAE}
\email{rt1986@nyu.edu}


\date{June 6, 2026}

\begin{document}

\begin{abstract}
In a recent work [\emph{Adv. Math.} {\bf 401} (2022), Paper No. 108318], a central limit theorem was established for the linear combinations of the iterates of a non-rotational inner function fixing the origin. In this paper, we prove the law of iterated logarithm (LIL) in the same setup, with a very mild condition on the coefficients. We also identify the full set of subsequential limit points at the LIL scale. Using the Aleksandrov-Clark decomposition and measure-preserving properties of the inner functions, one can construct a reverse martingale that is close to the linear combinations of inner functions. We prove the LIL for the partial sums of reverse martingale differences under a Feller-type assumption, which then transfers to the linear combinations of iterates of the inner functions. 
\end{abstract}
\thanks{The authors are sincerely grateful to Professor Artur Nicolau for his valuable comments and insightful suggestions, which greatly improved the paper.}

\maketitle

\section{Introduction}
\label{sec:introduction}
A large part of probability is concerned with limit theorems, which quantify the behaviour of a sequence of random variables at different scales. Let $(X_i)_{i\geq 1}$ be a sequence of independent and identically distributed (i.i.d.) random variables with mean $\mu$ and variance $\sigma^2$ and let $S_n=\sum_{i=1}^{n}X_i$. The law of large numbers (LLN) and the central limit theorem (CLT) quantify the behaviour of $(S_n-n\mu)$ at the scale $n$ and $\sqrt{n}$, respectively. More precisely, as $n\to \infty$, we have
\begin{align*}
    \frac{S_n-n\mu}{n}\to 0 \text{ almost surely}; \qquad 
    \frac{S_n-n\mu}{\sigma\sqrt{n}}\stackrel{d}{\Longrightarrow} \mathcal{N}(0, 1).
\end{align*}
Here, $\mathcal{N}(0, 1)$ denotes a standard Gaussian random variable and $\stackrel{d}{\Longrightarrow}$ denotes the convergence in the sense of distributions. Notice, however, at the CLT scale, the random variable $(S_n-n\mu)$ can and will take arbitrarily large values infinitely often. It is natural to ask, what happens in between the LLN and the CLT? This is answered by the law of the iterated logarithm (LIL), which states that 
\begin{align*}
\limsup_{n\to \infty }\frac{S_n-n\mu}{\sigma \sqrt{2n\log\log n}} =+1, \qquad \liminf_{n\to \infty }\frac{S_n-n\mu}{\sigma \sqrt{2n\log\log n}} =-1\;.
\end{align*}

Perhaps surprisingly, this question was first asked in the context of number theory. Borel's celebrated result that almost every number in $[0, 1]$ is normal is an instance of the law of large numbers. This inspired the study of finer limiting properties of the sum $\sum_{i=1}^{n}X_i$ where $X_i$ is the $i$th digit in the base $d$ expansion of a uniform random number $U\in [0, 1]$ by Hausdorff, Hardy and Littlewood, and Khintchine. In particular, Khintchine proved the famous law of the iterated logarithm for $S_n$. It is now well understood that the LLN, the CLT, and the LIL are not limited to the sum of i.i.d. random variables. In particular, the central limit theorem and the law of the iterated logarithm are known for several important classes of stochastic and dynamical systems, including independent random variables, martingales, Markov processes, random multiplicative functions, expanding dynamical systems, and certain holomorphic dynamical systems. The literature surrounding this is too vast to enumerate here. We refer the interested readers to~\cite{Stout74, Bingham, HH80}. Let us mention that Bingham in his excellent survey~\cite[Section 18--20]{Bingham} lists over forty different contexts of LIL in probability, statistics, analysis, and number theory. 

In analysis, one of the earliest and most prominent appearances of CLT and LIL is in the theory of lacunary trigonometric sums. It is long understood that lacunary trigonometric sums behave like the sum of i.i.d. random variables in several ways. Indeed, Kac~\cite{Kac49} and the Salem--Zygmund, Erdos--Gal, and Weiss~\cite{SalemZygmund1950, ErdosGal1955, Weiss1959} established the central limit theorem, and the laws of the iterated logarithm, respectively, for the lacunary exponential sums. We refer the reader to~\cite{AistleitnerBerkesTichy2023} for the history and motivations for studying these problems. We should also point out that the central limit theorem and the law of iterated logarithms are not valid for arbitrary trigonometric polynomials. It also naturally leads to the question: how much of this sum of i.i.d. random variables' behavior persists if $x\mapsto e^{2\pi k ix}$ is replaced by other $1$-periodic functions. Such questions have also been investigated in the recent past, and the effect of the arithmetic structure of the sequence $(n_k)_{k\geq 1}$ on the sum $S_n(x)=\sum_{k=1}^{n}f(n_k x)$ -- where $f$ is a $1$-periodic function-- has also been intensely investigated. These investigations reveal deep and intimate connections with the arithmetical properties of the sequence $(n_k)$. We refer the reader to phenomena\cite{AistleitnerGantertKabluchkoProchnoRamanan2023, AistleitnerFruehwirthHaukeManskova2025, AistleitnerBerkesTichy2023} and the references therein for more detail on this topic, as well as for more recent investigations into moderate and large deviation phenomena for lacunary trigonometric polynomials.

Lacunary trigonometric polynomials are also an important class of examples of Bloch functions that arise naturally in the study of the mappings of the disk. In a celebrated work, Makarov~\cite{Makarov} proved the LIL for Bloch functions and used it to quantify the distortion of the boundary of a Jordan domain under conformal mappings. We also refer to the series of works by Nicolau and Llorente~\cite{Nicolau2018, llorente2004regularity, llorente2026law} for more recent LIL-type results with applications to regularity properties of measures and functions
We also refer the reader to~\cite[Section 20]{Bingham} for early references in this area.

Many lacunary trigonometric polynomials arise naturally as partial sums of iterates of \emph{inner functions}. For the convenience of the reader, we recall the definition of an inner function in the next section. Introduced by Nevanlinna and developed through the pioneering work of the Riesz brothers, Frostman, and Beurling, inner functions have become a fundamental object in complex analysis, with deep connections to operator theory, harmonic analysis, and dynamical systems. We refer the reader to \cite{Garnett2007, Mashreghi} for more details.
The dynamics of iterates of inner functions have been studied extensively from the viewpoint of ergodic theory; see, for instance,~\cite{AaronsonNadkarni2023, Aaronson1978, FerreiraNicolau2024, IvriiUrbanski2023, Ivrii2026, NicolauSoler2022, Pommerenke1981}. More recently, Nicolau and Soler-i-Gibert established a central limit theorem for the linear combinations of the iterates of inner functions using a Fourier-analytic argumnt~\cite{NicolauSoler2022, NicolauSoler2026}. A natural problem is whether these sums admit a law of the iterated logarithm. The main purpose of the present paper is to establish such a theorem.

\section{Setup and main results}
\label{sec:Setup}
Let $\D$ denote the open unit disc and let $\T=\partial\D$ be the unit circle. Let $m$ denotes normalized Lebesgue measure on $\T$.  By Fatou's theorem, every bounded holomorphic function $f:\D\to\mathbb{C}$ admits radial limits
\[
\lim_{r\to1^-}f(r\zeta)
\]
for $m$-almost every $\zeta\in\T$, where $m$ denotes the normalized Lebesgue measure on $\T$. A bounded holomorphic function $f$ is called an \emph{inner function} if its radial boundary values have modulus one for $m$-almost every point of $\T$. We use $T_f:\T\to\T$ for its boundary map on the unit circle, and we also call $T_f$ an inner function.

Power maps $z\mapsto z^p$ and the finite Blaschke products are the basic examples of inner functions. For a function $g$ on the unit circle or on the unit disk, we denote by $g^n$ the $n$-fold composition of $g$, that is, we set $g^1=g$ and $g^n=g\circ g^{n-1}$ for $n\geq 1$. Let $Z$ denote the coordinate map $Z(z)=z$ on $\T$. We set $T_f^{0}=Z$. 

Let $b=(b_n)_{n\geq 0}$ be a sequence of complex numbers and $f$ be an inner function fixing the origin which is not a rotation. We are interested in studying the fluctuations of 
\[
S_N^{(b)}(z) := \sum_{n=0}^{N-1}b_n\,T_f^{n}(z)\;.
\]
In a recent work, Nicolau and Soler i Gibert~\cite{NicolauSoler2022} showed that $S_N^{(b)}$ behaves like the sum of independent, identically distributed random variables at the CLT scale. More precisely, let 
\[
\sigma_N^2(f) := \|S_{N}^{(b)}\|_{L^2(m)}^2
= \sum_{n=0}^{N-1}|b_n|^2
 +2\Ree\sum_{k=1}^{N-1}\lambda^k
       \sum_{n=0}^{N-k-1}\overline{b_n}b_{n+k},
\qquad \lambda=f'(0)\;.
\]
They showed that if 
\begin{equation}
\label{eqn:NSiG-CLT_condition}
    B_N:=\sum_{n=0}^{N-1}|b_n|^2\to \infty, \quad \text{and}\quad \frac{\max_{0\leq n<N}|b_n|^2}{B_{N}}\to 0,  
\end{equation}
as $N\to \infty$, then for any Borel measurable set $A\subseteq \mathbb{C}$, we have
\[
m\left\{z:\sqrt{\frac{2}{\sigma_N^2(f)}}S_N^{(b)}(z)\in A\right\}
\xrightarrow[N\to\infty]{}
\frac{1}{2\pi}\int_A e^{-|w|^2/2}\,dw\;.
\]
The original proof of CLT in~\cite{NicolauSoler2022} used a slightly different condition on the coefficients. The condition~\eqref{eqn:NSiG-CLT_condition} was shown in~\cite{NicolauSoler2026}, where the authors also showed that these conditions are necessary.

It is natural to ask if LIL holds for $S_N^{(b)}$. We remind the reader that for $p\geq 2$, power functions $z\mapsto z^p$ satisfy our assumptions, and for such functions the LIL follows from the work of Salem and Zygmund~\cite{SalemZygmund1950}. The LIL is generally harder to prove and requires stringent conditions (even for the sum of independent random variables). We prove LIL for $S_N^{(b)}$ under a very mild strengthening of the condition~\eqref{eqn:NSiG-CLT_condition}, analogous to Feller's condition for the LIL for the sum of independent random variables. We need the following assumption on the coefficient $(b_{n})_{n\geq 0}$. 
\begin{equation}
\label{eqn:LILcondition}
 B_{N}=\sum_{n=0}^{N-1}|b_n|^2\to \infty, \quad \text{and }\quad
 \frac{\max_{0\le n<N}|b_n|^2 L(L(B_N))}{B_N}\to 0\;,
\end{equation}
as $N\to \infty$. Here and throughout the paper, to simplify our notation and avoid some pathological issues, we set $L(x):=\log(\max\{e, x\})$. Note that $L(L(x))=\log\log x$ for $x\geq e^e$. Under assumption~\eqref{eqn:LILcondition}, we not only prove LIL but also identify the full set of subsequential limit points of $S_N^{(b)}$ at the LIL scale. Before we state our main result, we make one more definition. For a sequence $(x_n)_{n\in \mathbb{N}}\subseteq \C$, set $\Cl\{x_n:n\ge 1\}$ to denote the set of subsequential limits of $(x_n)_{n\in \mathbb{N}}$. 

\begin{theorem}[LIL for iterates of inner functions]
\label{thm:GeneralLIL}
Let $f$ be an inner function on $\mathbb D$ with $f(0)=0$ that is not a rotation. Let $(b_n)_{n\geq 0}$ be a sequence of complex numbers satisfying~\eqref{eqn:LILcondition}. Then, for $m$-almost every $z\in \T$, we have 
\[
\Cl\left\{
 \frac{S_N^{(b)}(z)}{\sqrt{\sigma_N^2(f)L(L(\sigma_N^2(f)))}} : N\ge 1 \right\} =\overline{\mathbb{D}}.
\]
\end{theorem}

The proof of Theorem~\ref{thm:GeneralLIL} can be roughly split into steps. The first step is to recognize that $S_N^{(b)}$ is close to the partial sums of a reverse martingale difference sequence (see Section~\ref{MRM} for definitions) $M_N$. Consequently, one can transfer the LIL for $M_N$ to an LIL for $S_N$. This step crucially relies on the Aleksandrov--Clark decomposition and the measure-preserving properties of inner functions (see \cref{subsection_AC}). The observation that $S_N^{(b)}$ behaves like the partial sums of a reverse martingale difference was already made by Chen, Fu, Lin, and Qiu~\cite{ChenFuLinQiu2026} to deduce a quantitative CLT for $S_N^{(b)}$ from the CLT for reverse martingale. It is convenient to rewrite the assumption~\eqref{eqn:LILcondition} and Theorem~\ref{thm:GeneralLIL-martingale} in a different but equivalent form, which we do in Theorem~\ref{thm:GeneralLIL-martingale}. This formulation is better suited to the reverse-martingale approach.

The second step is to establish a law of the iterated logarithm for the partial sums of reverse martingale differences. Our second main result establishes the LIL for $\mathbb{R}^2$-valued reverse martingale and identifies the full cluster set under appropriate assumptions. This is of independent interest in its own. While the LIL for reverse martingales and the partial sums of the reverse martingale differences is a rich and active area, there are several results available in this direction~\cite{ScottHuggins1983, Cuny2015, CunyMerlevede2015}. We find that none of the existing results give the full LIL for the reverse martingales under our assumptions. 

\begin{theorem}[LIL for the partial sums of reverse martingale differences]
\label{thm:LILforReverseMartingale}
Let $(\mathcal F_n)_{n\ge0}$ be a decreasing filtration on a probability space
$(\Omega,\mathcal F,\mathbb P)$, and let $(Y_n,\mathcal F_n)_{n\ge0}$ be
$\mathbb R^2$-valued reverse martingale differences, that is,
\[
        Y_n\text{ is }\mathcal F_n\text{-measurable},\qquad
        \E(Y_n\mid\mathcal F_{n+1})=0 .
\]
Assume that $|Y_n|\le c_n$ a.s. for deterministic constants $c_n\ge0$, and put
\[
        M_N=\sum_{n=0}^{N-1}Y_n,
        \qquad
        Q_N=\sum_{n=0}^{N-1}\E(Y_nY_n^T\mid\mathcal F_{n+1}),
        \qquad
        c_N^*=\max_{0\le n<N}c_n .
\]
Let $v_N>0$ be deterministic and nondecreasing, with $v_N\to\infty$, and assume
\begin{equation}\label{eq:revLIL-cov-matrix}
        \frac{Q_N}{v_N^2}\longrightarrow \frac12 I_2\quad\text{a.s.},
        \qquad
        \frac{c_N^*\sqrt{L(L(v_N^2))}}{v_N}\longrightarrow0 .
\end{equation}
Then, almost surely,
\[
        \Cl\left\{
        \frac{M_N}{v_N\sqrt{L(L(v_N^2))}}:N\ge1
        \right\}=\bigl\{x\in\mathbb R^2: |x|\le1\bigr\}.
\]
\end{theorem}

We now record an important corollary of Theorem~\ref{thm:GeneralLIL}. When $b_i=1$ for all $i\geq 0$. The conditions in ~\eqref{eqn:LILcondition} are readily verified. It is also easy to check in this case that $LL(B_N)\sim LL(N)$ and $\sigma_N^2(f)/B_N\to \sigma_f^2$ where
\[
  \sigma_f^2=\Ree\frac{1+\lambda}{1-\lambda}
    =\frac{1-|\lambda|^2}{|1-\lambda|^2}, \qquad \lambda = f'(0)\;.
\]
This yields the following corollary.

\begin{corollary}
\label{cor:main}
Let $f$ be an inner function on the disc such that $f(0)=0$ and $f$ is not a rotation. Let $T_f$ be its
boundary map. Put
\[
    S_N(z)=\sum_{n=0}^{N-1}T_f^n(z).
    .
\]
Then, for $m$-almost every $z\in\T$,  and for every $u\in\T$, we have

\begin{enumerate}
    \item[(a)] \[
        \Cl\left\{
        \frac{S_N(z)}{\sqrt{N\log\log N}}:N\ge 3
        \right\}
        =\{w\in\C: |w|\le \sigma_f\};
\]
\item[(b)] \[
         \limsup_{N\to\infty}
         \frac{|S_N(z)|}{\sqrt{N\log\log N}}
         =\sigma_f
         \qquad\text{for }m\text{-a.e. }z\in\T;
 \]
 \item[(c)]
\[
        \limsup_{N\to\infty}
        \frac{\Ree(\overline u S_N(z))}{\sqrt{2N\log\log N}}
        =\frac{\sigma_f}{\sqrt 2},
        \qquad
        \liminf_{N\to\infty}
        \frac{\Ree(\overline u S_N(z))}{\sqrt{2N\log\log N}}
        =-\frac{\sigma_f}{\sqrt 2}.
\]

\end{enumerate}
\end{corollary}

Let us mention the deep and remarkable work of Ivrri and Urba\'nski in a complementary direction. Ivrii and Urba\'nski~\cite{IvriiUrbanski2023} used the powerful thermodynamic formalism to analyze the dynamics of inner functions. They consider the partial sums of $h\circ f^n$ where $f$ is an inner function and $h$ is a suitable observable (a H\"older or a Sobolev function). Ivrii and Urba\'nski mention that their results, combined with the results of Gordin~\cite{Gor69} and Gou\"ezel~\cite{Gou10}, respectively, give the central limit theorem and the law of iterated logarithm for such partial sums. Taking $h(z)=z$, one obtains Corollary~\ref{cor:main}. 

Let us also mention that for the constant coefficient case, the associated reverse martingale differences that we construct (see~\eqref{eqn:SNvsMN}) are stationary. To get LIL for the partial sums of the stationary reverse martingale difference sequences, one can directly invoke~\cite{CunyMerlevede2015} instead of Theorem~\ref{thm:LILforReverseMartingale}, thus providing a short and quick proof in this case. 

We close this section with the discussion of another special case of Theorem~\ref{thm:GeneralLIL}. In the following discussion, it will be convenient to identify the interval $[0, 1)$ with the unit circle $\T$ via the map $\theta\mapsto z=e^{2\pi i\theta}$.  Let $f_p(z)=z^p$. When $p\geq 2$, the linear combinations of the iterates of $f_p$ 
\[L_N^{(b), p}(z):=\sum_{n=0}^{N}b_nf^{n}_p(z) = \sum_{n=0}^{N-1}b_n\exp(2\pi i p^n \theta)\]
are lacunary trigonometric sums for which the LIL was studied by Salem--Zygmund~\cite{SalemZygmund1950}, Erd\H{o}s-Gal~\cite{ErdosGal1955}, and Mary Weiss~\cite{Weiss1959}. We should point out that our conditions~\eqref{eqn:LILcondition} are precisely the assumption in ~\cite[Equation (1.2)]{Weiss1959}. Let us note, however, that Weiss~\cite{Weiss1959} proved the LIL for a more general class of lacunary trigonometric polynomials.


\subsection{Outline of the paper}
The paper is organized as follows. In \cref{sec:preliminaries}, we collect the preliminary results and definitions needed for the main proofs. 

In \cref{sec:ReverseMartingaleLIL}, we prove Theorem~\ref{thm:LILforReverseMartingale}. The proof is technical and long and relies on several Lemmas that are also proved in the same section. 

Section \ref{Sec:LILForInnerFunction} deals with the proof of Theorem~\ref{thm:GeneralLIL}. In Theorem~\ref{thm:GeneralLIL-martingale}, we provide an equivalent formulation, proof of which follows from Theorem~\ref{thm:LILforReverseMartingale} and a strong law of large numbers (SLLN) for the weighted sums of reverse martingale differences (Lemma~\ref{lem:SLLNWeighted}).

\section{Some preliminaries}\label{sec:preliminaries}

\subsection{Aleksandrov--Clark measures and decomposition}\label{subsection_AC}
Associated with every analytic self-map $f$ of the unit disc is a family of positive Borel measures on the unit circle, known as the Aleksandrov--Clark measures of $F$. They arise from the Herglotz representations of the functions $$
z\rightarrow\frac{\alpha+F(z)}{\alpha-F(z)},$$ 
for $\alpha\in\mathbb T$. Introduced by Clark \cite{Clark} and subsequently studied in depth by Aleksandrov \cite{AL1, AL2, AL3}, these measures have become an indispensable tool in function theory, complex analysis, and operator theory. We refer the reader to the survey article \cite{poltoratski2006aleksandrov} and to \cite[Chapter IX]{cima2006cauchy} for comprehensive accounts of their properties and applications. Since our focus is on inner functions, we restrict our attention to this setting and recall the corresponding description of Aleksandrov--Clark measures. We also collect several facts concerning the dynamics of inner functions that will be used throughout the paper.

Let $F$ be an inner function with $F(0)=0$.  For $\alpha\in\T$, let $\mu_\alpha$
be the Aleksandrov--Clark measure of $F$, defined as
\[
        \frac{\alpha+F(w)}{\alpha-F(w)}
        =\int_{\T}\frac{\zeta+w}{\zeta-w}\,d\mu_\alpha(\zeta),
        \qquad w\in\D. 
\]
Here, the possible imaginary constant in the Herglotz formula is zero because $F(0)=0$. Setting $w=0$ and using that $F(0)=0$, we immediately see that $\mu_{\alpha}$ is a probability measure, that is, $\mu_\alpha(\T)=1$. Moreover,
$\mu_\alpha$ is carried by $F^{-1}(\{\alpha\})$ for a.e. $\alpha\in \T$. A remarkable result of Aleksandrov shows that $(\mu_{\alpha})_{\alpha\in \T}$ is a disintegration of $m$ relative to $m$. More precisely, Aleksandrov's disintegration theorem says that
\begin{align}\label{eqn_theorem:AD}
\int_{\T} G(\zeta)\,dm(\zeta)
=
\int_{\T}\int_{\T} G(\zeta)\,d\mu_\alpha(\zeta)\,dm(\alpha).
\end{align}
for every $G\in L^1(\T,m)$. Let us point out that it is part of the remarkable result of Aleksandrov that $$\alpha\mapsto \int_{\T}G(\zeta)d\mu_{\alpha}(\zeta)$$ exists for a.e. $\alpha\in \T$ and is in $L^1(\T,m)$.

While we do not have the space to describe several interesting applications of Aleksandrov's disintegration, we record some consequences that will be used throughout the paper. The first consequence of Aleksandrov's disintegration theorem is that any inner function fixing the origin preserves the normalized Lebesgue measure $m$ on $\T$. That is, for any Borel set $A\subseteq \T$, we have $m(T^{-1}A)=m(A)$. 
Recall that a measurable map \(g:\mathbb T\to\mathbb T\) preserving
the normalized Lebesgue measure \(m\) is called \emph{ergodic with respect to $m$} if
whenever a measurable set \(A\subseteq \mathbb T\) satisfies
$g^{-1}(A)=A,$ then
\[
m(A)=0
\qquad \text{or} \qquad
m(A)=1.
\]
It is also an easy consequence of the disintegration theorem that an inner function fixing the origin, which is not a rotation, is ergodic with respect to $m$. The proof is standard, see, for example,~\cite{AaronsonBook, Pommerenke1981}. 

We now record some moment identities that are useful later. Writing
\[
        \frac{\alpha+F(w)}{\alpha-F(w)}
        =1+2\sum_{k\ge 1}\overline\alpha^{\,k} F(w)^k, \qquad 
        \frac{\zeta+w}{\zeta-w}
        =1+2\sum_{k\ge 1}\overline\zeta^{\,k} w^k, 
\]
in the disintegration formula~\eqref{eqn_theorem:AD} and
comparing the coefficients of $w^{k}$, we get moment identities. For an inner function $F$ fixing the origin $F(0)=0$, we also use the fact that near the origin, we have
\[
F(w) = F'(0)w+ \frac{F''(0)}{2}w^2 + O(w^3)\;.
\]
Using this, we obtain the following moment identities for the first two moments that we record below for later use. The proofs are standard and can be also found in~\cite{NicolauSoler2022}.
\begin{lemma}\label{lemma_measure}
Let $F$ be an inner function with $F(0)=0$, and let $\mu_\alpha$ denote its
Aleksandrov--Clark measure at $\alpha\in\mathbb T$. Then
\begin{align}
 \int_{\T}\zeta\,d\mu_\alpha(\zeta)
 &=\alpha\overline{F'(0)}, \label{eqn:FirstMoment}\\
 \int_{\T}\zeta^2\,d\mu_\alpha(\zeta)
 &=\alpha^2\overline{F'(0)}^{\,2}
   +\frac{\alpha\overline{F''(0)}}{2}. \label{eqn:SecondMoment}
\end{align}
\end{lemma}
\begin{proof}
Comparing the coefficients of $w$ and $w^2$ in the two expansions above gives
\[
 \int_{\T}\overline\zeta\,d\mu_\alpha(\zeta)
 =\overline\alpha F'(0),
 \qquad
 \int_{\T}\overline\zeta^{\,2}\,d\mu_\alpha(\zeta)
 =\overline\alpha^{\,2}F'(0)^2
  +\frac{\overline\alpha F''(0)}2.
\]
Taking complex conjugates proves the identities.
\end{proof}
We also need the following identity that shows that $X_n$ and $X_m$ are exponentially decorrelated. 
\begin{lemma}
    Let $T_F$ denote the boundary map of $F$, let $X_n=T_F^n$, and put
$\lambda=F'(0)$. Then, for $n<m$,
\begin{equation}\label{inner_cauchy}
 \int_{\mathbb T}X_n(z)\,\overline{X_m(z)}\,dm(z)
 =\overline\lambda^{\,m-n}.
\end{equation}
\end{lemma}
\begin{proof}
Indeed, invariance of $m$ reduces the integral to
\[\int_{\T}Z\,\overline{T_F^{m-n}}\,dm=\overline{(F^{(m-n)})'(0)}=\overline\lambda^{\,m-n}\;.\]
\end{proof}

We now record the conditional expectation identities. For completeness, we give the proof (also see~\cite{ChenFuLinQiu2026}). We use $\mathcal{B}$ to denote the Borel sigma algebra on $\T$.
\begin{lemma}\label{lem:regression}
Let $f$ be an inner function on $\mathbb D$ with $f(0)=0$, let $T_f$ be its
boundary map, and put $\lambda=f'(0)$ and $a_2=f''(0)/2$. Let $Z(z)=z$.
Then
\begin{enumerate}
 \item $\E[Z\mid T_f^{-1}\Borel]=\overline\lambda T_f$;
 \item $\E[Z^2\mid T_f^{-1}\Borel]=\overline\lambda^{\,2}T^2+\overline{a_2}T_f$.
\end{enumerate}
\end{lemma}
\begin{proof}
It is enough to test against functions $\psi\circ T_f$, with
$\psi\in L^\infty(\T_f)$. By Aleksandrov--Clark disintegration and
\eqref{eqn:FirstMoment},
\begin{align*}
 \int_{\T}Z(z)\psi(T_f(z))\,dm(z)
 &=\int_{\T}\psi(\alpha)
   \left(\int_{\T}\zeta\,d\mu_\alpha(\zeta)\right)dm(\alpha)\\
 &=\overline\lambda\int_{\T}\alpha\psi(\alpha)\,dm(\alpha)\\
 &=\int_{\T}\overline\lambda T_f(z)\psi(T_f(z))\,dm(z),
\end{align*}
where the last equality uses the invariance of $m$. This proves (1).
Likewise, \eqref{eqn:SecondMoment} gives
\begin{align*}
 \int_{\T}Z(z)^2\psi(T_fz)\,dm(z)
 &=\int_{\T}\psi(\alpha)
   \left(\alpha^2\overline\lambda^{\,2}
         +\alpha\overline{a_2}\right)dm(\alpha)\\
 &=\int_{\T}
   \left(\overline\lambda^{\,2}T_f(z)^2+\overline{a_2}T_f(z)\right)
   \psi(Tz)\,dm(z).
\end{align*}
This proves (2).
\end{proof}

\subsection{Martingales and reverse martingale differences}\label{MRM}
We collect some basic identities that will be useful in showing that $S_N^{(b)}$ is close to a martingale. We begin with the standard definitions for completeness. These definitions can be found in any standard graduate probability textbook.  See, for example,
\cite[Chapters~10--12]{Williams1991}.
An increasing (decreasing) filtration is a
sequence of $\sigma$-algebras $(\mathcal F_n)_{n\ge0}$ satisfying
$\mathcal F_0\subseteq\mathcal F_1\subseteq\cdots$ (
$\mathcal F_0\supseteq\mathcal F_1\supseteq\cdots$). A sequence of random variables $(X_n)_{n\geq 0}$ is adapted to a sequence of sigma-algebras $(\mathcal{F}_n)_{n\geq 0}$ if $X_n$ is $\mathcal{F}_n$ measurable for each $n$. A sequence $(X_n)$ of random variables adapted to an increasing (decreasing) filtration is a martingale (reverse martingale) if
$\E(X_{n+1}\mid\mathcal F_n)=X_n$ ($\E(X_{n}\mid\mathcal F_{n+1})=X_{n+1}$). A sequence of random variables $(D_n)_{n\geq 0}$ adapated to an (decreasing) increasing filtration $(\mathcal{F}_n)$ is a (reverse) martingale-difference sequence if ($\E[D_n\mid\mathcal{F}_{n+1}]=0$) $\E[D_n\mid \mathcal F_{n-1}]=0$. The partial sums of a martingale difference sequence give a martingale. However, the partial sums of the reverse martingale difference sequences do not form a reverse martingale.

Put $X_n=T^n$ and let $\mathcal F_n=\sigma(X_n)$ be the sigma algebra generated by $X_n$. It is easy to see that $(\mathcal{F}_n)_{n\geq 0}$ is decreasing filtration. Set $D_n=X_n-\overline\lambda X_{n+1}$. As an immediate consequence of Lemma~\ref{lem:regression}, we get that $D_n$ is a reverse martingale difference sequence. This observation was already made in~\cite{ChenFuLinQiu2026}.

\begin{lemma}\label{cor:MomentsOfDn}
For every $n\ge0$,
\begin{align*}
 \E[D_n\mid\mathcal F_{n+1}]&=0,\\
 \E[|D_n|^2\mid\mathcal F_{n+1}]&=1-|\lambda|^2,\\
 \E[D_n^2\mid\mathcal F_{n+1}]&=\overline{a_2}X_{n+1}.
\end{align*}
\end{lemma}
\begin{proof}
Apply Lemma~\ref{lem:regression} after composing with $T^n$. The first and the third formulas follow by direct expansion. For the second equality, we use the fact that $|X_n|=|X_{n+1}|=1$ along with
$\E(X_n\overline{X_{n+1}}\mid\mathcal F_{n+1})=\overline\lambda$.
\end{proof}

We end this sequence with the following simple observation that follows immediately from Lemma~\ref{cor:MomentsOfDn}. 
\begin{corollary}\label{cor:reversemartingale}
For every deterministic complex sequence $(\beta_n)$, the variables
$Y_n=\beta_nD_n$ satisfy
\[
 Y_n\in L^\infty(\mathcal F_n),\qquad
 \E[Y_n\mid\mathcal F_{n+1}]=0.
\]
Thus $(Y_n,\mathcal F_n)$ is a reverse martingale-difference sequence, and we
write $M_N=\sum_{n=0}^{N-1}Y_n$ for its partial sums. If
$v_N^2=(1-|\lambda|^2)\sum_{n=0}^{N-1}|\beta_n|^2$, then
\begin{align}
 \frac1{v_N^2}\sum_{n=0}^{N-1}
 \E[|Y_n|^2\mid\mathcal F_{n+1}]&=1,
 \label{eqn:QuadraticVariation}\\
 \frac1{v_N^2}\sum_{n=0}^{N-1}
 \E[Y_n^2\mid\mathcal F_{n+1}]&=
 \frac{\overline{a_2}}{v_N^2}
 \sum_{n=0}^{N-1}\beta_n^2X_{n+1}.
 \label{eqn:QuadraticCov}
\end{align}
\end{corollary}

In Section~\ref{Sec:LILForInnerFunction}, we show that by choosing an appropriate sequence $(\beta_n)_{n\geq 0}$ depending on $(b_n)_{n\geq 0}$, the partial sum of reverse martingale differences $\sum_{n=0}^{N}\beta_nD_n$ is close to $S_N^{(b)}$. This allows us to transfer the LIL for the partial sums of the reverse martingale difference sequence to the LIL for $S_N^{(b)}$.

\section{LIL for the partial sums of reverse martingale differences}
\label{sec:ReverseMartingaleLIL}
In this section, we prove Theorem~\ref{thm:LILforReverseMartingale}. The proof is long and relies on several auxiliary lemmas that we state and prove in this section. Throughout this section $\langle\cdot,\cdot\rangle$ is the Euclidean inner
product on $\mathbb R^2$, and $A\preceq B$ denotes the Loewner order on
symmetric matrices.

\begin{lemma}[Bernstein bound for finite reverse blocks]
\label{lem:reverse-freedman}
Let $(\xi_n,\mathcal F_n)$ be a real reverse martingale
difference sequence, and assume $\xi_n\le c$ a.s. Put
\[
 S_{r,m}=\sum_{n=r}^{m-1}\xi_n,
 \qquad
 V_{r,m}=\sum_{n=r}^{m-1}\E(\xi_n^2\mid\mathcal F_{n+1}).
\]
Then, for $x,\sigma^2>0$,
\begin{equation}\label{eqn:BernsteinBound}
 \mathbb P\left\{
 \max_{r\le m\le s}S_{m,s}\ge x,\ V_{r,s}\le\sigma^2
 \right\}
 \le \exp\left(-\frac{x^2}{2(\sigma^2+cx/3)}\right).
\end{equation}
The analogous estimate holds for $-\xi_n$ when $\xi_n\ge-c$.
\end{lemma}
\begin{proof}
Set $\mathcal G_j=\mathcal F_{s-j}$ and $\eta_j=\xi_{s-j}$ for
$1\le j\le s-r$. Then $(\eta_j,\mathcal G_j)$ is an ordinary martingale-
difference sequence, and \eqref{eqn:BernsteinBound} is Freedman's maximal
Bernstein inequality~\cite[Theorem~1.6]{Freedman1975} applied to the reversed block.
\end{proof}

We now prove the upper LIL for the partial sums of a real-valued reverse martingale difference sequence. This may be of independent interest. 

\begin{proposition}[Upper LIL bound for real reverse martingale]
\label{lem:real-reverse-lil-upper}
Let $(\xi_n,\mathcal F_n)$ be a real reverse martingale difference sequence satisfying
$|\xi_n|\le b_n$ almost surely, where $b_n$ is deterministic. Put
\[
 S_N=\sum_{n=0}^{N-1}\xi_n,
 \qquad B_N=\sum_{n=0}^{N-1}\E(\xi_n^2\mid\mathcal F_{n+1}),
 \qquad b_N^*=\max_{0\le n<N}b_n.
\]
Let $s_N>0$ be deterministic and nondecreasing, with $s_N\to\infty$ and
$s_{N+1}/s_N\to1$. If
\[
 \frac{B_N}{s_N^2}\to1\quad\text{a.s.},
 \qquad
 \frac{b_N^*\sqrt{L(L(s_N^2))}}{s_N}\to0,
\]
then
\[
 \limsup_{N\to\infty}
 \frac{S_N}{\sqrt{2s_N^2L(L(s_N^2))}}\le1
 \qquad\text{a.s.}
\]
\end{proposition}
\begin{proof}
Fix $\rho>1$ and $\varepsilon,\delta>0$, and work on the full-measure set
where $B_N/s_N^2\to1$. Let $N_k$ be the least integer with
$s_{N_k}^2\ge\rho^k$. Then
\[
 s_{N_k}^2\sim\rho^k,
 \qquad d_k^2:=s_{N_{k+1}}^2-s_{N_k}^2\sim(\rho-1)\rho^k,
 \qquad L(L(s_{N_k}^2))\sim\log k.
\]
Put
\[
 x_k=(1+\varepsilon)\sqrt{2(1+\delta)s_{N_k}^2\log k},
 \qquad
 y_k=(1+\varepsilon)\sqrt{2(1+\delta)d_k^2\log k}.
\]
The assumption $b_N^{*}\sqrt{LL(s_N^2)}/s_N\to 0$ implies
\[
 \frac{b_{N_k}^*x_k}{s_{N_k}^2}\to 0,
 \qquad
 \frac{b_{N_{k+1}}^*y_k}{d_k^2}\to 0.
\]
Moreover,
\[
 B_{N_{k+1}}-B_{N_k}
 =d_k^2+o(s_{N_{k+1}}^2)+o(s_{N_k}^2)
 =d_k^2+o(\rho^k),
\]
so, eventually,
\[
 B_{N_k}\le(1+\delta)s_{N_k}^2,
 \qquad
 B_{N_{k+1}}-B_{N_k}\le(1+\delta)d_k^2.
\]
Using Lemma~\ref{lem:reverse-freedman}, first on $[0,N_k)$ and then for
$-\xi_n$ on $[N_k,N_{k+1})$, we obtain
\begin{align*}
 \mathbb P\{S_{N_k}\ge x_k,\ B_{N_k}\le(1+\delta)s_{N_k}^2\}
 &\le k^{-1-\eta_1},\\
 \mathbb P\left\{
 \max_{N_k\le m<N_{k+1}}(S_m-S_{N_{k+1}})\ge y_k,
 B_{N_{k+1}}-B_{N_k}\le(1+\delta)d_k^2
 \right\}
 &\le k^{-1-\eta_2}
\end{align*}
for some $\eta_1,\eta_2>0$ and all large $k$. By Borel--Cantelli, if
$N_k\le N<N_{k+1}$, then eventually
$S_N\le x_{k+1}+y_k$. Consequently,
\[
 \limsup_{N\to\infty}
 \frac{S_N}{\sqrt{2s_N^2L(L(s_N^2))}}
 \le(1+\varepsilon)\sqrt{1+\delta}
       (\sqrt\rho+\sqrt{\rho-1}).
\]
Letting $\varepsilon\downarrow0$, $\delta\downarrow0$, and then
$\rho\downarrow1$ proves the assertion.
\end{proof}

In the proof of Theorem~\ref{thm:LILforReverseMartingale}, we apply Proposition~\ref{lem:real-reverse-lil-upper} to the one-dimensional projections $\xi_n=\langle u, \beta_nD_n\rangle$ over a countable dense collection of $u$ in $\T$. Using a standard continuity argument, this gives that the cluster set of $S_N^{(b)}/\sqrt{\sigma_N^2(f)LL(\sigma_N^2(f))}$ is contained in the closed unit disk. 

We now begin the preparation for proving the reverse containment. 

\begin{lemma}[Conditional exponential-tilt estimates]
\label{lem:conditional-tilt-estimates}
Let $X$ be an $\mathbb R^2$-valued random variable, let $\mathcal G$ be a
sigma-field, and assume
\[
 \E(X\mid\mathcal G)=0,\qquad |X|\le d.
\]
Put $A=\E(XX^T\mid\mathcal G)$ and
$\psi(\theta)=\log\E(e^{\langle\theta,X\rangle}\mid\mathcal G)$. There is a
universal constant $C$ such that, whenever $|\theta|d\le1/2$,
\begin{align*}
 \left|\psi(\theta)-\frac12\theta^TA\theta\right|
 &\le C|\theta|d\,\theta^TA\theta,\\
 \left|\nabla\psi(\theta)-A\theta\right|
 &\le Cd\,\theta^TA\theta,\\
 \tr\nabla^2\psi(\theta)&\le C\tr A.
\end{align*}
\end{lemma}
\begin{proof}
Write $Z=\langle\theta,X\rangle$ and $q=\E(Z^2\mid\mathcal G)
=\theta^TA\theta$. Since $\E(Z\mid\mathcal G)=0$ and $|Z|\le1/2$,
Taylor's formula gives
\[
 \E(e^Z\mid\mathcal G)=1+\frac q2+O(|\theta|d\,q).
\]
Taking the logarithm proves the first inequality. Similarly,
\[
 \E(Xe^Z\mid\mathcal G)=A\theta+O(dq),
\]
and division by $\E(e^Z\mid\mathcal G)\ge1$ proves the second inequality. Finally, $\nabla^2\psi(\theta)$ is the
conditional covariance of $X$ under the exponentially tilted conditional law.
Therefore
\[
 \tr\nabla^2\psi(\theta)
 \le\frac{\E(|X|^2e^Z\mid\mathcal G)}
          {\E(e^Z\mid\mathcal G)}
 \le e^{1/2}\tr A,
\]
which proves the last inequality.
\end{proof}
\begin{lemma}
\label{lem:completed-block}
Let $I_k=[r_k,s_k)$ be finite intervals and put
\[
 B_k=\sum_{n=r_k}^{s_k-1}Y_n,
 \qquad
 Q_{I_k}=\sum_{n=r_k}^{s_k-1}
 \E(Y_nY_n^T\mid\mathcal F_{n+1}),
 \qquad
 c_{I_k}=\max_{r_k\le n<s_k}c_n.
\]
Let $V_k>0$, $\Lambda_k\to\infty$, $a_k=V_k\sqrt{\Lambda_k}$, and let
$\delta_k\downarrow0$. Assume
\begin{equation}\label{eq:completed-block-small-increments}
 \frac{c_{I_k}\sqrt{\Lambda_k}}{V_k}\longrightarrow0.
\end{equation}
On an extension of the probability space carrying independent auxiliary
variables, one can construct
\[
 \widetilde B_k=\widehat B_k+E_k
\]
with the following properties.
\begin{enumerate}
 \item If
 $Q_{I_k}\preceq(\frac12+\delta_k)V_k^2I_2$, then
 $\widehat B_k=B_k$. If, in addition,
 $\|Q_{I_k}/V_k^2-I_2/2\|_{\mathrm{op}}\le\delta_k$, then, conditionally on
 $\mathcal F_{r_k}$, $E_k$ is a sum of bounded independent centered vectors
 whose total covariance matrix is
 \[
 R_k=\left(\frac12+\delta_k\right)V_k^2I_2-Q_{I_k},
 \qquad 0\preceq R_k\preceq2\delta_kV_k^2I_2.
 \]
 \item For every fixed $w\in\mathbb R^2$ and every $\eta>0$,
 \begin{equation}\label{eq:completed-block-lower}
 \widetilde{\mathbb P}\left(
 |\widetilde B_k-wa_k|\le\eta a_k\mid\mathcal F_{s_k}
 \right)
 \ge
 \exp\left\{-(|w|^2+2\eta|w|+o(1))\Lambda_k\right\},
 \end{equation}
 where the $o(1)$ is deterministic.
 \item On the event
 $\|Q_{I_k}/V_k^2-I_2/2\|_{\mathrm{op}}\le\delta_k$, for every $\eta>0$,
 \begin{equation}\label{eq:completion-tail}
 \widetilde{\mathbb P}\left(
 |E_k|>\eta a_k\mid\mathcal F_{r_k}
 \right)
 \le4\exp\left(
 -\frac{c\eta^2\Lambda_k}{\delta_k+\eta/\Lambda_k}
 \right)
 \end{equation}
 for a universal $c>0$ and all sufficiently large $k$.
\end{enumerate}
\end{lemma}
\begin{proof}
Fix $k$, write $m=s_k-r_k$, and reverse the original block by setting
\[
 \mathcal G_j=\mathcal F_{s_k-j},
 \qquad \zeta_j=Y_{s_k-j},
 \qquad 1\le j\le m.
\]
Put
\[
 A_j=\E(\zeta_j\zeta_j^T\mid\mathcal G_{j-1}),
 \qquad
 T_k=\left(\frac12+\delta_k\right)V_k^2I_2.
\]
Define predictably, with $C_0=0$,
\[
 h_j=\mathbf1_{\{C_{j-1}+A_j\preceq T_k\}},
 \qquad
 \widehat\zeta_j=h_j\zeta_j,
 \qquad
 C_j=C_{j-1}+h_jA_j.
\]
Then $(\widehat\zeta_j,\mathcal G_j)$ is a martingale-difference sequence and
$C_m\preceq T_k$. Put $R_k=T_k-C_m$. On an independent auxiliary space let
$U_1,U_2,\ldots$ be independent, each uniformly distributed on
\[
 \{\sqrt2e_1,-\sqrt2e_1,\sqrt2e_2,-\sqrt2e_2\}.
\]
Thus $\E U_\ell=0$, $\E(U_\ell U_\ell^T)=I_2$, and $|U_\ell|=\sqrt2$.
Let $q_k=\lceil\Lambda_k^3\rceil$ and, after the original reversed block,
append
\[
 \xi_\ell=q_k^{-1/2}R_k^{1/2}U_\ell,
 \qquad 1\le\ell\le q_k.
\]
More explicitly, after time $m$ use the increasing filtration
\[
 \mathcal G_m\vee\sigma(U_1,\ldots,U_\ell),
 \qquad 0\le\ell\le q_k.
\]
The matrix $R_k$ is $\mathcal G_m$-measurable, and therefore the $\xi_\ell$
are martingale differences with conditional covariance $q_k^{-1}R_k$.
Their total predictable covariance is $R_k$, so the total predictable
covariance of the augmented block is exactly $T_k$. Moreover, every augmented
increment is bounded by
\begin{equation}\label{eq:augmented-increment-bound}
 b_k^{\mathrm{aug}}:=\max\left\{c_{I_k},\frac{CV_k}{\Lambda_k^{3/2}}\right\},
 \qquad
 \frac{b_k^{\mathrm{aug}}\sqrt{\Lambda_k}}{V_k}\to0.
\end{equation}
Let $\widehat B_k=\sum_j\widehat\zeta_j$ and
$E_k=\sum_\ell\xi_\ell$. If $Q_{I_k}\preceq T_k$, every raw partial
covariance sum is bounded by $Q_{I_k}$, so $h_j=1$ for every $j$ and
$\widehat B_k=B_k$. On the stronger covariance event in the statement,
$R_k=T_k-Q_{I_k}$ and $0\preceq R_k\preceq2\delta_kV_k^2I_2$. This proves
part~(1).

We prove part~(2). List all augmented increments as $X_{\nu,k}$, let
$\mathcal H_{\nu,k}$ be their increasing filtration, and put
\[
 A_{\nu,k}=\E(X_{\nu,k}X_{\nu,k}^T\mid\mathcal H_{\nu-1,k}).
\]
Then $\sum_\nu A_{\nu,k}=T_k$. Write
$t_k=1/2+\delta_k$ and set
\[
 \theta_k=T_k^{-1}(wa_k)=\frac{w\sqrt{\Lambda_k}}{t_kV_k}.
\]
By \eqref{eq:augmented-increment-bound},
$|\theta_k|b_k^{\mathrm{aug}}\to0$. Define
\[
 \psi_{\nu,k}(\theta)=
 \log\E(e^{\langle\theta,X_{\nu,k}\rangle}
       \mid\mathcal H_{\nu-1,k}),
 \qquad
 K_k(\theta)=\sum_\nu\psi_{\nu,k}(\theta).
\]
Lemma~\ref{lem:conditional-tilt-estimates} gives, uniformly in the underlying
sample point,
\begin{align}
 K_k(\theta_k)
 &=\frac12\theta_k^TT_k\theta_k+o(\Lambda_k)
 =\frac{|w|^2}{2t_k}\Lambda_k+o(\Lambda_k),
 \label{eq:completed-cumulant}\\
 m_k^{\theta}:=\sum_\nu\nabla\psi_{\nu,k}(\theta_k)
 &=T_k\theta_k+o(a_k)=wa_k+o(a_k),
 \label{eq:completed-drift}\\
 \sum_\nu\tr\nabla^2\psi_{\nu,k}(\theta_k)&\le CV_k^2.
 \label{eq:completed-variance}
\end{align}
Here, the error in \eqref{eq:completed-drift} is bounded by
$Cb_k^{\mathrm{aug}}\Lambda_k=o(a_k)$.

The density process
\[
 L_{j,k}=\exp\left\{
 \left\langle\theta_k,\sum_{\nu\le j}X_{\nu,k}\right\rangle
 -\sum_{\nu\le j}\psi_{\nu,k}(\theta_k)
 \right\}
\]
is a positive martingale with $L_{0,k}=1$. Thus, for a bounded $H$, define
\[
 \widetilde{\mathbb E}^{\theta_k}
 (H\mid\mathcal F_{s_k})
 =\widetilde{\mathbb E}(HL_{\mathrm{fin},k}
   \mid\mathcal F_{s_k}).
\]
Under this tilted conditional law,
$X_{\nu,k}-\nabla\psi_{\nu,k}(\theta_k)$ is a martingale difference. Hence,
orthogonality of its increments and \eqref{eq:completed-variance} gives
\[
 \widetilde{\mathbb E}^{\theta_k}\left(
 \left|\widetilde B_k-m_k^{\theta}\right|^2
 \mid\mathcal F_{s_k}\right)\le CV_k^2.
\]
Chebyshev's inequality, therefore, yields
\[
 \widetilde{\mathbb P}^{\theta_k}\left(
 |\widetilde B_k-m_k^{\theta}|>\frac{\eta a_k}{2}
 \mid\mathcal F_{s_k}\right)=O(\Lambda_k^{-1}).
\]
Together with \eqref{eq:completed-drift}, the tilted conditional probability
of $\{|\widetilde B_k-wa_k|\le\eta a_k\}$ tends to one. On this event,
\[
 \langle\theta_k,\widetilde B_k\rangle
 \le\langle\theta_k,wa_k\rangle+|\theta_k|\eta a_k.
\]
Since
\[
 L_{\mathrm{fin},k}^{-1}
 =\exp\{-\langle\theta_k,\widetilde B_k\rangle
          +K_k(\theta_k)\},
\]
changing measure back and using \eqref{eq:completed-cumulant} yields
\begin{align*}
 &\widetilde{\mathbb P}\left(
 |\widetilde B_k-wa_k|\le\eta a_k\mid\mathcal F_{s_k}\right)\\
 &\quad\ge
 \exp\left\{-\frac{|w|^2+2\eta|w|}{1+2\delta_k}\Lambda_k
              -o(\Lambda_k)\right\}(1-o(1)),
\end{align*}
which implies \eqref{eq:completed-block-lower}. The argument also covers
$w=0$, with $\theta_k=0$.

Finally, on the covariance event in part~(3), each auxiliary increment has
norm at most $CV_k/\Lambda_k^{3/2}$ and the total conditional covariance is
bounded by $2\delta_kV_k^2I_2$. Conditionally on $\mathcal F_{r_k}$, apply the ordinary forward form of
Freedman's inequality for both signs of each coordinate of $E_k$. A union
bound gives \eqref{eq:completion-tail}.
\end{proof}

The final ingredient in the proof of Theorem~\ref{thm:LILforReverseMartingale} is a Borel--Cantelli lemma. We need a version slightly different from~\cite{Free73}, which is stated for increasing filtration, but the argument can be readily adapted to obtain the next lemma. For concreteness, we supply a short argument. 
\begin{lemma}[A Borel--Cantelli lemma]
\label{lem:reverse-conditional-BC}
Let $(\mathcal H_k)$ be decreasing and let $A_k\in\mathcal H_{k-1}$. If
$p_k=\mathbb P(A_k\mid\mathcal H_k)$, then
\[
 \left\{\sum_kp_k=\infty\right\}\subseteq\{A_k\ \text{infinitely often}\}
 \quad\text{a.s.}
\]
\end{lemma}
\begin{proof}
Fix $m\le n$ and ignore the null set on which $p_k=1$ but $A_k$ fails. On
$\bigcap_{k=m}^n\{p_k<1\}$ put
\[
 R_{m,n}=\prod_{k=m}^n\frac{1-\mathbf1_{A_k}}{1-p_k}.
\]
The product of the factors with index larger than $k$ is
$\mathcal H_k$-measurable, and
\[
 \E\left(\frac{1-\mathbf1_{A_k}}{1-p_k}\,\middle|\,\mathcal H_k\right)=1.
\]
Iterated conditioning, therefore, gives $\E R_{m,n}=1$. On the event that no
$A_k$ occurs for $m\le k\le n$,
\[
 R_{m,n}=\prod_{k=m}^n(1-p_k)^{-1}
 \ge \exp\left(\sum_{k=m}^np_k\right).
\]
Consequently, Markov's inequality gives, for $t>0$,
\[
 \mathbb P\left(A_k^c\text{ for }m\le k\le n,
                 \ \sum_{k=m}^np_k\ge t\right)\le e^{-t}.
\]
If no $A_k$ occurs after $m$ and $\sum_{k\ge m}p_k=\infty$, the event inside
the probability holds for all sufficiently large $n$. Fatou's lemma bounds its
probability by $e^{-t}$ for every $t>0$, hence by zero. Taking the union over
$m$ proves the result.
\end{proof}

\subsection{Proof of Theorem~\ref{thm:LILforReverseMartingale}}
Put
\[
 a_N=v_N\sqrt{L(L(v_N^2))},
 \qquad q_N=\tr Q_N.
\]
First observe that $v_{N+1}/v_N\to1$. Indeed, $q_N/v_N^2\to1$ a.s. and
\[
 0\le q_{N+1}-q_N
 =\E(|Y_N|^2\mid\mathcal F_{N+1})
 \le c_N^2=o(v_{N+1}^2).
\]
Thus $q_N/v_{N+1}^2\to1$, and comparison with $q_N/v_N^2\to1$ proves the
claim.

Fix a unit vector $u\in\mathbb R^2$ and set
$\xi_n=\langle u,Y_n\rangle$. Then
\[
 \sum_{n=0}^{N-1}\E(\xi_n^2\mid\mathcal F_{n+1})
 =u^TQ_Nu\sim\frac12v_N^2.
\]
Lemma~\ref{lem:real-reverse-lil-upper}, with $s_N=v_N/\sqrt2$, gives
\[
 \limsup_{N\to\infty}\frac{\langle u,M_N\rangle}{a_N}\le1
 \quad\text{a.s.}
\]
For each $j$, choose a finite $2^{-j}$-net $\mathcal D_j$ of the unit circle
and intersect the corresponding full-measure sets. If $M_N\ne0$, choose
$u_N\in\mathcal D_j$ with $|u_N-M_N/|M_N||\le2^{-j}$. Then
\[
 \left(1-2^{-2j-1}\right)\frac{|M_N|}{a_N}
 \le \max_{u\in\mathcal D_j}\frac{\langle u,M_N\rangle}{a_N}.
\]
Taking upper limits and then letting $j\to\infty$ gives
\begin{equation}\label{eq:upper-cluster-contained}
 \limsup_{N\to\infty}\frac{|M_N|}{a_N}\le1,
 \qquad
 \Cl\{M_N/a_N:N\ge1\}\subseteq\overline{\mathbb D}
 \quad\text{a.s.}
\end{equation}

For the reverse inclusion, fix $w\in\mathbb R^2$ with $0<|w|<1$ and choose
$1<\alpha<|w|^{-2}$. Let $N_k$ be the least integer such that
$v_{N_k}^2\ge e^{k^\alpha}$. Since $v_{N+1}/v_N\to1$,
\[
 v_{N_k}^2\sim e^{k^\alpha},
 \qquad \Lambda_k:=L(L(v_{N_k}^2))=\alpha\log k+o(\log k),
 \qquad \frac{v_{N_{k-1}}}{v_{N_k}}\to0.
\]
Put
\[
 I_k=[N_{k-1},N_k),\qquad V_k=v_{N_k},\qquad
 B_k=M_{N_k}-M_{N_{k-1}}.
\]
Since $Q_{I_k}=Q_{N_k}-Q_{N_{k-1}}$,
\[
 \frac{Q_{N_{k-1}}}{V_k^2}
 =\frac{Q_{N_{k-1}}}{v_{N_{k-1}}^2}
  \frac{v_{N_{k-1}}^2}{v_{N_k}^2}\longrightarrow0
 \quad\text{a.s.}
\]
Consequently,
\[
 \frac{Q_{I_k}}{V_k^2}\longrightarrow\frac{1}{2}I_2\quad\text{a.s.},
 \qquad
 \frac{c_{I_k}\sqrt{\Lambda_k}}{V_k}
 \le\frac{c_{N_k}^*\sqrt{L(L(v_{N_k}^2))}}{v_{N_k}}
 \longrightarrow0.
\]
Since $\frac{Q_{I_k}}{V_k^2}\to \frac{1}{2}I_2$ almost surely, we can find a deterministic $\delta_k\downarrow 0$ such that
\[
\mathbb{P}(\left\|Q_{I_k}/V_k^2-I_2/2\right\|_{\mathrm{op}}>\delta_k)\leq 2^{-k}.
\]
In particular, the Borel--Cantelli lemma gives us that
\begin{equation}\label{eq:block-envelope}
 \left\|Q_{I_k}/V_k^2-I_2/2\right\|_{\mathrm{op}}\le\delta_k
 \quad\text{eventually a.s.}
\end{equation}
Apply Lemma~\ref{lem:completed-block} to all blocks on one product extension,
using independent auxiliary families from block to block. Let
$\mathcal U_k$ be the sigma-field generated by the auxiliary variables in
block $k$, and define the decreasing filtration
\[
 \mathcal H_k=\mathcal F_{N_k}\vee\sigma(\mathcal U_j:j>k).
\]
For $\eta>0$, set
\[
 A_k(w,\eta)=\{|\widetilde B_k-wV_k\sqrt{\Lambda_k}|
                   \le\eta V_k\sqrt{\Lambda_k}\}.
\]
Then $A_k(w,\eta)\in\mathcal H_{k-1}$, and independence of the auxiliary
families from the original space and from one another give
\[
 \widetilde{\mathbb P}(A_k(w,\eta)\mid\mathcal H_k)
 =\widetilde{\mathbb P}(A_k(w,\eta)\mid\mathcal F_{N_k}).
\]
By \eqref{eq:completed-block-lower}, if $\eta>0$ is so small that
\[
 \beta:=\alpha(|w|^2+3\eta|w|)<1,
\]
then the last conditional probability is at least $k^{-\beta}$ eventually.
Lemma~\ref{lem:reverse-conditional-BC} therefore gives
\begin{equation}\label{eq:augmented-hits-io}
 A_k(w,\eta)\quad\text{infinitely often},
 \qquad \widetilde{\mathbb P}\text{-a.s.}
\end{equation}

On the full-measure event in \eqref{eq:block-envelope}, part~(1) of
Lemma~\ref{lem:completed-block} gives $\widehat B_k=B_k$ eventually. Moreover,
\eqref{eq:completion-tail} and $\Lambda_k\sim\alpha\log k$ imply, for every
fixed $\varepsilon>0$,
\[
 \sum_k\widetilde{\mathbb P}\left(
 |E_k|>\varepsilon V_k\sqrt{\Lambda_k}
 \mid\mathcal F_{N_{k-1}}\right)<\infty
 \quad\text{a.s.}
\]
Indeed, the coefficient of $\Lambda_k$ in the exponent of
\eqref{eq:completion-tail} tends to infinity. Fix an original sample point in
the full-measure event \eqref{eq:block-envelope}. For that sample point, the
probabilities, with respect to the auxiliary coordinates, are summable.
The Borel--Cantelli lemma on the auxiliary product space, followed by
Fubini's theorem, therefore, yields
\begin{equation}\label{eq:completion-negligible}
 \frac{|E_k|}{V_k\sqrt{\Lambda_k}}\longrightarrow0
 \qquad\widetilde{\mathbb P}\text{-a.s.}
\end{equation}
Combining \eqref{eq:augmented-hits-io} and
\eqref{eq:completion-negligible}, we obtain
\[
 |B_k-w a_{N_k}|\le(\eta+o(1))a_{N_k}
 \quad\text{for infinitely many }k.
\]
By \eqref{eq:upper-cluster-contained},
$|M_{N_{k-1}}|=O(a_{N_{k-1}})$, and
$a_{N_{k-1}}/a_{N_k}\to0$. Hence, along those indices,
\[
 \left|\frac{M_{N_k}}{a_{N_k}}-w\right|\le\eta+o(1).
\]
Let $\eta\downarrow0$ through a countable sequence. Intersecting over a countable dense subset of the punctured open disk shows, on the product extension, that the cluster set contains the closed unit disk. This completes the proof.
\qed

\section{LIL for inner functions}\label{Sec:LILForInnerFunction}
We now prove Theorem~\ref{thm:GeneralLIL}. We recall that $f$ is an inner function fixing the origin and which is not a rotation. We use $T=T_f$ to denote the boundary map of $f$ and put $\lambda = f'(0)$. Since $f$ is not a rotation, Schwarz's lemma gives $|\lambda|<1$. For notational convenience, we have
\[
 X_n=T^n,\qquad D_n=X_n-\overline\lambda X_{n+1},
 \qquad \mathcal F_n=\sigma(X_n).
\]
Recall from Lemma~\ref{cor:MomentsOfDn} that $D_n$ is a reverse martingale difference sequence with respect to the filtration $\mathcal{F}_n$.  Given a sequence $(b_n)$ of complex numbers, we now define an auxiliary sequence
\[
 \beta_{-1}=0,
 \qquad
 \beta_n=\sum_{j=0}^n\overline\lambda^{\,n-j}b_j,
 \qquad n\ge0.
\]
Since $b_n=\beta_n-\overline\lambda\beta_{n-1}$, summation by parts gives
\begin{equation}\label{eqn:SNvsMN}
 S_N^{(b)}=M_N+\overline\lambda\beta_{N-1}X_N,
 \qquad
 M_N=\sum_{n=0}^{N-1}\beta_nD_n.
\end{equation}
Note that $M_N$ is a partial sum of the reverse martingale difference sequence. Further notice that 
\[
\sigma_N^2(f) = \E[|M_N|^2]+ |\lambda|^2|\beta_{N-1}|^2 = \sum_{n=0}^{N-1}|\beta_N^2||D_n|^2+|\lambda|^2|\beta_{N-1}|^2\;.
\]
Put
\begin{equation*}
 A_N=\sum_{n=0}^{N-1}|\beta_n|^2,
 \qquad
 v_N^2=(1-|\lambda|^2)A_N,
 \qquad
 \Delta_N=\max_{0\le n<N}|\beta_n|.
\end{equation*}
Recall that $B_N=\sum_{n=0}^{N-1}|b_n|^2$ and $b_N^*=\max_{0\le n<N}|b_n|$. Set $r=|\lambda|<1$ and note that
 \begin{align*}
  \frac{B_N}{2(1+r^2)}\leq A_N&\le\frac{B_N}{(1-r)^2},\qquad  \frac{b_N^{*}}{1+r}\leq  \Delta_N\le\frac{b_N^*}{1-r}. 
 \end{align*}
 Therefore, the condition~\eqref{eqn:LILcondition} in Theorem~\ref{thm:GeneralLIL} is equivalent to
 \begin{equation}\label{eqn:LILCondition-re}
  A_N\to\infty,
  \qquad
  \frac{\Delta_N^2L(L(A_N))}{A_N}\to0.
 \end{equation}
Furthermore, under assumptions~\eqref{eqn:LILcondition} (equivalently~\eqref{eqn:LILCondition-re}) we also have that $\sigma_N^2(f)\sim v_N^2$. 
This implies that 
\[
\Cl\left\{
 \frac{S_N^{(b)}(z)}{\sqrt{\sigma_N^2(f)L(L(\sigma_N^2(f)))}}:N\ge1
 \right\}= \Cl\left\{
 \frac{S_N^{(b)}(z)}{\sqrt{v_N^2L(L(v_N^2))}}:N\ge1
 \right\}\;.
\]
Therefore, we can reformulate Theorem~\ref{thm:GeneralLIL} in the following equivalent form, which is more convenient for us. 

\begin{theorem}[LIL for inner functions]
\label{thm:GeneralLIL-martingale}
Let $f$ be an inner function on $\mathbb D$ with $f(0)=0$ that is not a
rotation. Let $(b_n)$ be complex coefficients such that $(\beta_n)$ satisfy~\eqref{eqn:LILCondition-re}. Then, for $m$-almost every $z\in\T$,
\[
 \Cl\left\{
 \frac{S_N^{(b)}(z)}{\sqrt{v_N^2L(L(v_N^2))}}:N\ge1
 \right\}=\overline{\mathbb D}.
\]
\end{theorem}
\begin{proof}
We begin by observing that $S_N^{(b)}$ and $M_N$ have the same cluster set at the LIL scale. That is, note that 
\[
\frac{|S_N^{(b)}-M_N|}
 {v_N\sqrt{L(L(v_N^2))}}= \frac{|\overline\lambda\beta_{N-1}X_N|}
 {v_N\sqrt{L(L(v_N^2))}}
 \le C_f\frac{\Delta_N}{\sqrt{A_NL(L(A_N))}}\longrightarrow 0.
\]
It remains to apply Theorem~\ref{thm:LILforReverseMartingale} to the complex
variables $Y_n=\beta_nD_n$, identified with vectors in $\mathbb R^2$. They satisfy
\[
 |Y_n|\le c_n:= (1+|\lambda|)|\beta_n|\;.
\]
The second condition in~\eqref{eq:revLIL-cov-matrix} follows directly from~\eqref{eqn:LILCondition-re}. Let 
\[
  Q_N=\sum_{n=0}^{N-1}\E(Y_nY_n^T\mid\mathcal F_{n+1}).
\]
It remains to check that $\frac{Q_n}{v_N^2}\to \frac{1}{2}I$ almost surely, as $N\to \infty$. To this end, we first note that it is equivalent to showing that 
\begin{align*}
    \frac{1}{v_N^2}\sum_{n=0}^{N-1}\E[|Y_n|^2\mid \mathcal{F}_{n+1}] \to 1,\qquad \frac{1}{v_N^2}\sum_{n=0}^{N-1}\E(Y_n^2\mid\mathcal F_{n+1}) \to 0, 
\end{align*}
almost surely, as $N\to \infty$. 
The first condition follows immediately from Corollary~\ref{cor:reversemartingale}, which in fact gives
\[
\frac{1}{v_N^2}\sum_{n=0}^{N-1}\E[|Y_n|^2\mid \mathcal{F}_{n+1}] =1, \text{almost surely}\;.
\]
On the other hand, Corollary~\ref{cor:reversemartingale} also gives
\[
 \frac{1}{v_N^2}\sum_{n=0}^{N-1}\E(Y_n^2\mid\mathcal F_{n+1}) = \frac{\overline{a_2}}{v_N^2}\sum_{n=0}^{N-1}\beta_n^2X_{n+1}\;,
\]
which goes to $0$ almost surely by Lemma~\ref{lem:SLLNWeighted} (with $a_n=\beta_n^2$
$W_N=A_N$, $a_N^*=\Delta_N^2$) that we prove below.

\end{proof}

We state and prove the last assertion as a separate Lemma because we believe it can be of independent interest. It can be seen as a strong law of large numbers (SLLN) for the weighted sum $\sum_{n=0}^{N-1}a_nX_{n+1}$ under Tsuchikura type assumption~\cite{Tsuchi51}. This is closely related to a result of Cohen and Lin~\cite{CohenLin09}. 

The almost sure convergence of weighted sums of random variables is an intensely investigated topic that began around 1950~\cite{Tsuchi51, MR13435, JOP, Heyde, MR221571}. Kolmogorov's three series theorem is a cornerstone result in the field pertaining to the case of weighted sums of independent random variables. In our case, the $X_n$ are not independent. However, we can write $\sum_{n=0}^{N-1}a_nX_n$ as the partial sums of reverse martingale differences, up to a bounded error term.  Our conditions are essentially Tsuchikura's condition~\cite{Tsuchi51}. We note that the strong law of large numbers for the weighted sum of the martingale difference sequences under Tsuchikura's condition was proved by Cohen and Lin~\cite[Corollary 4.12]{CohenLin09} where it is attributed to Azuma~\cite{MR221571}. Since we are working with the weighted sum of a reverse martingale difference sequence as opposed to a martingale difference sequence, we find it prudent to give a complete proof, which is an adaptation of Azuma's original proof.

\begin{lemma}[SLLN for weighted sums of reverse martingale differences]
\label{lem:SLLNWeighted}
Let $(a_n)$ be deterministic complex numbers, and put
\[
 W_N=\sum_{n=0}^{N-1}|a_n|,
 \qquad a_N^*=\max_{0\le n<N}|a_n|,
 \qquad T_N=\sum_{n=0}^{N-1}a_nX_{n+1}.
\]
Assume $W_N\to\infty$ and $ \frac{a_N^*L(L(W_N))}{W_N}\to0$. Then, $T_N/W_N\to0$ almost surely.
\end{lemma}
\begin{proof}
Define
\[
 \gamma_n=\sum_{j=0}^n\overline\lambda^{\,n-j}a_j,
 \qquad \gamma_{-1}=0.
\]
Then $a_n=\gamma_n-\overline\lambda\gamma_{n-1}$ and
\begin{equation}\label{eq:weighted-martingale-decomposition}
 T_N=\sum_{n=0}^{N-1}\gamma_nD_{n+1}
     +\overline\lambda\gamma_{N-1}X_{N+1}.
\end{equation}
Writing $r=|\lambda|$, Young's inequality gives
\begin{equation}
\label{eq:gamma-bounds}
 \Gamma_N:=\max_{n<N}|\gamma_n|\le\frac{a_N^*}{1-r},
 \qquad
 \sum_{n=0}^{N-1}|\gamma_n|^2
 \le\frac{a_N^*W_N}{(1-r)^2}.
\end{equation}
The first term in \eqref{eq:weighted-martingale-decomposition} is a complex
partial sum of reverse martingale differences. Applying
Lemma~\ref{lem:reverse-freedman} to its real and imaginary parts and both
signs, and using \eqref{eq:gamma-bounds}, yields a constant $C_f$ such that
for every $x>0$,
\begin{equation}\label{eq:weighted-freedman}
 \mathbb P\left(
 \left|\sum_{n=0}^{N-1}\gamma_nD_{n+1}\right|>x\right)
 \le4\exp\left(
 -\frac{x^2}{C_f(a_N^*W_N+a_N^*x)}
 \right).
\end{equation}
The same estimate holds for the maximum of the backward tails over any finite
block.

Fix $\rho>1$ and let $N_k$ be the least index with $W_{N_k}\ge\rho^k$.
By our assumption, we have $a_N^*/W_N\to 0$. Since
$W_{N_k-1}<\rho^k\le W_{N_k}$ and the last jump is at most $a_{N_k}^*$, we have $W_{N_k}\sim\rho^k$, and also
\[
 \frac{W_{N_k}}{a_{N_k}^*\log k}\longrightarrow\infty.
\]
Taking $x=\varepsilon\rho^k$ in \eqref{eq:weighted-freedman}, both at
$N_{k+1}$ and for the maximal backward tail on $[N_k,N_{k+1})$, gives
probabilities bounded by
\[
 4\exp\left(-c_\varepsilon
             \frac{W_{N_{k+1}}}{a_{N_{k+1}}^*}\right),
\]
which are summable. Hence, Borel--Cantelli gives,
for all sufficiently large $k$,
\[
 \left|\sum_{n=0}^{N_{k+1}-1}\gamma_nD_{n+1}\right|
 \le\varepsilon\rho^k
\]
and
\[
 \max_{N_k\le N<N_{k+1}}
 \left|\sum_{n=N}^{N_{k+1}-1}\gamma_nD_{n+1}\right|
 \le\varepsilon\rho^k.
\]
Since $W_N\ge\rho^k$ for $N_k\le N<N_{k+1}$, subtraction of the backward
tail and then letting $\varepsilon\downarrow0$ through a countable sequence
shows, uniformly on these blocks, that
\[
 \frac{1}{W_N}\sum_{n=0}^{N-1}\gamma_nD_{n+1}\longrightarrow0
 \qquad\text{a.s.}
\]
Finally,
$|\lambda\gamma_{N-1}|/W_N\le r a_N^*/((1-r)W_N)\to0$, so
\eqref{eq:weighted-martingale-decomposition} completes the proof.

\end{proof}

\subsection*{Funding} 
The first author was supported by the Slovenian Research Agency program P1-0222 and grant J1-50002. This work was performed
within the project COMPUTE, funded within the QuantERA II Programme, which has received funding from the European Union's Horizon 2020 research and innovation program under Grant Agreement No.~101017733.


\bibliographystyle{amsalpha}
\bibliography{references}

\end{document}